\documentclass[pdflatex,sn-basic]{sn-jnl}

\usepackage{graphicx}%
\usepackage{multirow}%
\usepackage{amsmath,amssymb,amsfonts,mathtools}%
\usepackage{amsthm}%
\usepackage{mathrsfs}%
\usepackage[title]{appendix}%
\usepackage[x11names]{xcolor}%
\usepackage{textcomp}%
\usepackage{manyfoot}%
\usepackage{booktabs}%
\usepackage{algorithm}%
\usepackage{algorithmicx}%
\usepackage{algpseudocode}%
\usepackage{listings}%
\usepackage{enumitem}%

\usepackage[nameinlink]{cleveref}
\hypersetup{colorlinks=true,allcolors=Blue4}

\theoremstyle{thmstyleone}%
\newtheorem{theorem}{Theorem}
\newtheorem{proposition}[theorem]{Proposition}%

\theoremstyle{thmstyletwo}%
\newtheorem{remark}{Remark}%

\theoremstyle{thmstylethree}%

\AddToHook{env/theorem/begin}{\crefalias{theorem}{theorem}}
\AddToHook{env/proposition/begin}{\crefalias{theorem}{proposition}}
\AddToHook{env/example/begin}{\crefalias{example}{example}}
\AddToHook{env/remark/begin}{\crefalias{remark}{remark}}
\AddToHook{env/definition/begin}{\crefalias{definition}{definition}}

\DeclareMathOperator{\id}{id}
\DeclareMathOperator*{\argmin}{arg\,min}
\DeclareMathOperator{\conv}{conv}

\begin{document}

\title{A Convex Feasibility Problem for Which Douglas--Rachford Converges Only Weakly}


\author*[1]{\fnm{Jan} \sur{Quan}}\email{jan.quan@kuleuven.be}

\author[1]{\fnm{Panagiotis} \sur{Patrinos}}\email{panos.patrinos@kuleuven.be}

\affil[1]{\centering \orgdiv{STADIUS}, \orgname{KU Leuven}, \orgaddress{\street{Kasteelpark Arenberg 10}, \city{Leuven}, \postcode{3001}, \country{Belgium}}}


\abstract{
    We provide two closed, convex sets with nonempty intersection and an initial point for which the Douglas--Rachford algorithm converges weakly but not strongly. This resolves an open question as noted earlier by \citet*[Remark~2.2]{aragon2014recent}. Our construction is a simple lift of a recent construction by Bauschke and Tung concerning firmly nonexpansive operators.
}




\maketitle

\section{Introduction}
In this paper, we consider the feasibility problem
\begin{equation*} \label{eq:problem} \tag{P}
    \textnormal{Find } x \in A \cap B,
\end{equation*}
where $A,B \subseteq \mathcal{H}$ are closed, convex sets with nonempty intersection, and $\mathcal{H}$ is a real Hilbert space with inner product $\langle\cdot\mid\cdot\rangle$ and induced norm $\|\cdot\|$.

The Douglas--Rachford (DR) method is a classical approach for solving \eqref{eq:problem}. It was first developed by \citet{douglas1956numerical} for the solution of linear equations arising from heat conduction problems and later generalized by \citet{lions1979splitting} to finding the zeros of the sum of two maximally monotone operators. The latter formulation applies directly to \eqref{eq:problem}, with the two operators chosen as the normal cones associated with the sets $A$ and $B$. Denote by $P_C : x \mapsto \argmin_{p\in C} \|x - p\|$ the (necessarily single-valued) projector onto a nonempty, closed, convex set $C\subseteq\mathcal{H}$. Further, let $R_C = 2 P_C - \id$ be the reflector on $C$ where $\id$ denotes the identity operator on $\mathcal{H}$. Then, the DR algorithm for solving \eqref{eq:problem} is given by the iteration
\begin{equation*} \label{eq:dr} \tag{DR}
    y_{n+1} = T_{A,B}(y_n), \quad \quad T_{A,B} = \frac12(\id + R_BR_A) = \id - P_A + P_B(2P_A - \id),
\end{equation*}
with corresponding shadow sequences $\{P_A(y_n)\}$ and $\{P_BR_A(y_n)\}$, where $n\in\mathbb{N}\coloneq \{0, 1, 2, \ldots\}$.

We recall some basic results, for which proofs can be found in \citet{lions1979splitting}, \citet{svaiter2011weak}, and \citet{bauschke2017correction}.
\begin{theorem} \label{thm:weak_dr}
    Let $\mathcal{H}$ be a real Hilbert space and $A,B\subseteq\mathcal{H}$ be two closed, convex sets with $A \cap B \neq \emptyset$. Consider the Douglas--Rachford operator $T_{A,B}$ and iterates $\{y_n\}$ defined in \eqref{eq:dr}. Then, there exists a fixed point $y\in\mathcal{H}$ of $T_{A,B}$ such that
    \begin{enumerate}[label=(\roman*)]
        \item $y_n\rightharpoonup y$,
        \item $P_A(y_n) \rightharpoonup P_A(y)$, $P_BR_A(y_n) \rightharpoonup P_A(y)$, and $P_A(y) \in A \cap B$.
    \end{enumerate}
    Here, $\rightharpoonup$ denotes weak convergence, that is, $x_n \rightharpoonup x$ if $\langle x_n\mid z\rangle \to \langle x\mid z\rangle$ for all $z\in\mathcal{H}$.
\end{theorem}
It is a natural question to ask whether the weak convergence in \Cref{thm:weak_dr} can be improved to strong convergence; see for example \citet*[Remark~2.2]{aragon2014recent}. In the more general setting of finding a zero of the sum of two maximally monotone operators, a counterexample to strong convergence was constructed by \citet{bui2020douglas}, building on a counterexample for alternating projections given by \citet{hundal2004alternating}. We provide a counterexample in the setting of \Cref{thm:weak_dr} by an elementary lift of a construction introduced by \citet{bauschke2026ces} concerning firmly nonexpansive operators. Concurrently, \citet*{bui2026douglas} constructed a different counterexample in the setting of \Cref{thm:weak_dr}, based on the cone and associated hyperplane described in \citet{hundal2004alternating}. At first glance, it seems that our construction is simpler.

\begin{remark} \label{rem:bui_feasibility}
    The counterexample in \citet{bui2020douglas} does not apply to the setting considered in \Cref{thm:weak_dr}. Indeed, suppose both operators in \citet[Counterexample~2]{bui2020douglas} were normal cones. Then the resolvent in \citet[Eq.\,(1)]{bui2020douglas} equals the projector $P_C$ onto some set $C\subseteq \mathcal{H}$ and Eq.\,(11) in \citet{bui2020douglas} implies that $y_{n+1} = P_C(y_n)$. Therefore, $P_C(y_{n+1}) =P_C^2(y_n) = P_C(y_n)$, hence $\{P_C(y_n)\}$ is a constant sequence, contradicting its failure to converge strongly in \citet[Counterexample~2]{bui2020douglas}.
\end{remark}

\section{Main result}
\citet{bauschke2026ces} constructed a firmly nonexpansive operator $T$ with $0$ as its unique fixed point, together with an initial point so that the fixed-point iteration converges weakly but not strongly. Moreover, the iterates are constructed to attain equality in the firm nonexpansiveness inequality with respect to $0$. Recall that an operator $T : \mathcal{H} \to \mathcal{H}$ is firmly nonexpansive if $\langle x - y \mid T(x) - T(y)\rangle \geq \|T(x) - T(y)\|^2$ for all $x,y\in\mathcal{H}$ \citep[see][Proposition~4.4]{bauschke2017correction}. Formally, we have the following result.
\begin{proposition} \label{prop:fne_cntex}
    There exists a real Hilbert space $\tilde{\mathcal{H}}$, a firmly nonexpansive operator $T : \tilde{\mathcal{H}} \to \tilde{\mathcal{H}}$ with (unique) fixed point $0$ and a point $z_0 \in \tilde{\mathcal{H}}$ such that, denoting $z_n = T^n(z_0)$ for each $n\in\mathbb{N}$, we have
    \begin{equation} \label{eq:weakstr_fne}
        z_n \rightharpoonup 0, \qquad \|z_n\| \downarrow \rho_\infty \coloneq \inf_{m\in\mathbb{N}}\|z_m\| > 0,
    \end{equation}
    \begin{equation} \label{eq:fne_eq}
        (\forall n \in\mathbb{N})\quad  \langle z_n \mid z_{n+1} \rangle = \langle z_n - 0 \mid T(z_n) - T(0)\rangle = \|T(z_n) - T(0)\|^2 = \|z_{n+1}\|^2.
    \end{equation}
\end{proposition}
\begin{proof}
    See \citet[Theorem~7.1 and Proposition~6.2]{bauschke2026ces}.
\end{proof}

We now provide a natural lift of this counterexample to the product space, which exhibits a counterexample to the strong convergence of \eqref{eq:dr} in the setting of \Cref{thm:weak_dr}. It is natural in the sense that, after fixing $A$ to a trivial choice, \eqref{eq:dr} determines certain points that must belong to $B$. We then take the smallest set containing these points and satisfying the desired properties.

\begin{theorem} \label{thm:dr_countex}
    Let the Hilbert space $\tilde{\mathcal{H}}$, firmly nonexpansive operator $T : \tilde{\mathcal{H}} \to \tilde{\mathcal{H}}$, initial point $z_0 \in \tilde{\mathcal{H}}$ and $\rho_\infty = \inf_{m\in\mathbb{N}} \|z_m\| > 0$ be furnished by \Cref{prop:fne_cntex}, where we denote $z_n = T^n(z_0)$ for $n\in\mathbb{N}$. Define
    \[
        \mathcal{H} = \tilde{\mathcal{H}} \times \tilde{\mathcal{H}}, \qquad A = \tilde{\mathcal{H}} \times \{0\}, \qquad B = \overline{\conv}(\{0\} \cup \{(z_{n+1}, z_{n+1} - z_n) : n \in \mathbb{N}\}),
    \]
    where $\overline{\conv}(C)$ denotes the closed convex hull of $C\subseteq\mathcal{H}$ \citep[see][Definition~3.3]{bauschke2017correction}.
    
    Choose $y_0 = (z_0, z_0)$ and consider the sequence $\{y_n\}$ generated by \eqref{eq:dr} with associated shadow sequences $\{P_A(y_n)\}$ and $\{P_BR_A(y_n)\}$. Then,
    \begin{enumerate}[label=(\roman*)]
        \item \label{item:dr_sets} $A$ and $B$ are closed and convex sets with nonempty intersection,
        \item \label{item:dr_iterates} $y_n = (z_n,z_n)$, $P_A(y_n) = (z_n, 0)$ and $P_BR_A(y_n) = (z_{n+1},z_{n+1}-z_n)$ for each $n\in\mathbb{N}$,
        \item \label{item:dr_convergence} $y_n \rightharpoonup 0$, $P_A(y_n) \rightharpoonup 0$, $P_BR_A(y_n) \rightharpoonup 0$, $\|y_n\| \downarrow \sqrt{2}\rho_\infty > 0$ and $\|P_A(y_n)\| = \|P_BR_A(y_n)\| \downarrow \rho_\infty > 0$.
    \end{enumerate}
\end{theorem}
\begin{proof}
    \ref{item:dr_sets} That $A,B$ are closed, convex and $0\in A\cap B$ is clear by construction.

    \ref{item:dr_iterates} We show that $y_n = (z_n,z_n)$ for each $n\in\mathbb{N}$ by induction. The base case $y_0 = (z_0, z_0)$ is due to the hypothesis. For the induction step, assume that $y_n = (z_n, z_n)$. By \eqref{eq:dr}, we have $y_{n+1} = y_n - P_A(y_n) + P_BR_A(y_n)$. That $P_A(y_n) = (z_n, 0)$ and $R_A(y_n) = (z_n, -z_n)$ follows immediately from the definition of $A$. Thus, the desired result follows if $P_B((z_n, -z_n)) = (z_{n+1}, z_{n+1} - z_n)$. To this end, note that $(z_{n+1}, z_{n+1} - z_n) \in B$ by construction, so in light of \citet[Theorem~3.16]{bauschke2017correction}, it remains to show that
    \begin{equation} \label{eq:proj_var}
        (\forall b \in B) \quad \langle (z_n, -z_n) - (z_{n+1}, z_{n+1} - z_n) \mid b - (z_{n+1}, z_{n+1} - z_n)\rangle \leq 0.
    \end{equation}
    By construction of $B$, we only need to check the preceding inequality for the points $b \in \{0, (z_1, z_1 - z_0), (z_{2}, z_2 - z_1), \ldots\}$. For $b=0$, we must show
    \[
      \langle (z_n - z_{n+1}, - z_{n+1}) \mid (-z_{n+1}, z_n - z_{n+1})\rangle \leq 0,
    \]
    which in fact holds with equality due to \eqref{eq:fne_eq}. Now let $m \in \mathbb{N}$ and $b = (z_{m+1}, z_{m+1} - z_m)$, so we must show that
    \[
         \textnormal{l.h.s.} \coloneq  \langle (z_n - z_{n+1}, -z_{n+1}) \mid (z_{m+1}, z_{m+1} - z_m) - (z_{n+1}, z_{n+1} - z_n)\rangle \leq 0.
    \]
    Expanding the left-hand side and using \eqref{eq:fne_eq} yields
    \begin{align*}
        \textnormal{l.h.s.} &= \langle z_n \mid z_{m+1} \rangle + \langle z_{n+1}\mid z_{m}\rangle - 2\langle z_{n+1}\mid z_{m+1}\rangle \\
        &= \|z_{n+1} - z_{m+1}\|^2 - \langle z_n - z_m\mid z_{n+1} - z_{m+1}\rangle.
    \end{align*}
    Finally, since $T$ is firmly nonexpansive by \Cref{prop:fne_cntex}, we have
    \[
        \langle z_n - z_m\mid z_{n+1} - z_{m+1}\rangle = \langle z_n - z_m\mid T(z_n) - T(z_m)\rangle \geq \|T(z_n) - T(z_m)\|^2 = \|z_{n+1} - z_{m+1}\|^2,
    \]
    which implies l.h.s.\ $\leq 0$, as desired. Thus, $y_n = (z_n, z_n)$ for each $n\in\mathbb{N}$.

    The statements $P_A(y_n) = (z_n, 0)$ and $P_BR_A(y_n) = (z_{n+1}, z_{n+1} - z_n)$ for each $n\in\mathbb{N}$ were already established above.
    
    \ref{item:dr_convergence} The weak-convergence assertions follow from \ref{item:dr_iterates} and $z_n\rightharpoonup 0$ in \eqref{eq:weakstr_fne}. The statements $\|y_n\| \downarrow \sqrt{2}\rho_\infty > 0$ and $\|P_A(y_n)\| \downarrow \rho_\infty > 0$ follow immediately from \ref{item:dr_iterates} and \eqref{eq:weakstr_fne}. Finally, by \ref{item:dr_iterates} and \eqref{eq:fne_eq}, we have
    \begin{align*}
        \|P_BR_A(y_n)\|^2 &= \|z_{n+1}\|^2+\|z_{n+1}-z_n\|^2 \\
        &= 2\|z_{n+1}\|^2 - 2\langle z_{n+1}\mid z_n\rangle + \|z_n\|^2 \\
        &= \|z_n\|^2\downarrow\rho_\infty^2>0.
    \end{align*}
\end{proof}

\begin{remark}
    In fact, $A,B$ in \Cref{thm:dr_countex} satisfy $A\cap B = \{0\}$. This follows from an argument similar to that in \citet[Lemma~6.4]{bauschke2026ces}. Concretely, due to the definition of $A$, points in $A\cap B$ are necessarily of the form $(y, 0)$ for some $y\in\tilde{\mathcal{H}}$, so it suffices to show that $(y, 0)\in B \implies y = 0$. Therefore, let $(y,0)\in B$ with $y\in\tilde{\mathcal{H}}$. It follows from \eqref{eq:proj_var} that $\langle (z_n - z_{n+1}, -z_{n+1})\mid (y - z_{n+1}, z_n - z_{n+1})\rangle \leq 0$ for each $n\in\mathbb{N}$, which is equivalent to $\langle z_n\mid y\rangle \leq \langle z_{n+1}\mid y\rangle$ for each $n\in\mathbb{N}$ in light of \eqref{eq:fne_eq}. By \eqref{eq:weakstr_fne}, we have $z_n\rightharpoonup0$, hence $\langle z_n\mid y\rangle\leq 0$ for each $n\in\mathbb{N}$. Since $y \in \overline{\conv}(\{0\} \cup \{z_{n+1} : n \in \mathbb{N}\})$, it follows that $\|y\|^2 \leq 0$, ensuring $y=0$ and $A\cap B=\{0\}$.
\end{remark}

\begin{remark} \label{rem:cesaro}
    Moreover, $T$ and $z_0$ in \Cref{thm:dr_countex} can even be chosen so that the Ces\`aro means of $\{y_n\}$ and both shadow sequences converge weakly but not strongly. Indeed, Example~7.4 in \citet{bauschke2026ces} provides a sequence $\{z_n\}$ as in \Cref{prop:fne_cntex} for which, denoting $\bar z_N = \frac1N\sum_{n=0}^{N-1}z_n$ for each $N\geq1$, we have $\inf_{N\geq1}\|\bar z_N\|>0$. Averaging the identities in \Cref{thm:dr_countex}\ref{item:dr_iterates} yields
    \[
        \bar y_N \coloneq \frac1N\sum_{n=0}^{N-1}y_n=(\bar z_N,\bar z_N), \qquad
        \frac1N\sum_{n=0}^{N-1}P_A(y_n)=(\bar z_N,0).
    \]
    By \eqref{eq:weakstr_fne}, we have $z_n\rightharpoonup0$, hence $\bar z_N\rightharpoonup0$. Thus, both sequences of Ces\`aro means converge weakly to $0$, while $\sqrt{2}\|\bar z_N\|$ and $\|\bar z_N\|$ are bounded away from $0$.
    Further, \eqref{eq:dr} yields
    \[
        \frac1N\sum_{n=0}^{N-1}P_BR_A(y_n)=\frac1N\sum_{n=0}^{N-1}P_A(y_n)+\frac{y_N-y_0}{N},
    \]
    and $(y_N-y_0)/N\to0$ since $\{y_n\}$ is bounded. Thus, the Ces\`aro means of $\{P_BR_A(y_n)\}$ also converge weakly but not strongly, as desired.
\end{remark}

\backmatter

\begingroup
\small
\section*{Declarations}

\noindent\textbf{Funding.} This work was supported by the Research Foundation Flanders (FWO) PhD grant 11A8T26N; Research Council KUL grant C14/24/103.
\par\smallskip
\noindent\textbf{Competing interests.} There are no financial or non-financial interests that are directly or indirectly related to this work.
\par\smallskip
\noindent\textbf{Data availability.} We do not analyse or generate any datasets, because our work proceeds within a theoretical and mathematical approach.
\par \smallskip
\noindent\textbf{Acknowledgements}
The authors acknowledge the assistance of ChatGPT 5.6 Sol for exploration of ideas. The main idea of lifting the construction of Bauschke and Tung is due to the authors. The authors thank Konstantinos Oikonomidis for helpful conversations.
\endgroup

\bibliography{references}

@article{aragon2014recent,
  title={Recent results on {D}ouglas--{R}achford methods for combinatorial optimization problems},
  author={Arag{\'o}n Artacho, Francisco J and Borwein, Jonathan M and Tam, Matthew K},
  journal={Journal of Optimization Theory and Applications},
  volume={163},
  number={1},
  pages={1--30},
  year={2014},
  doi={10.1007/s10957-013-0488-0},
  publisher={Springer}
}

@article{douglas1956numerical,
  title={On the numerical solution of heat conduction problems in two and three space variables},
  author={Douglas, Jim and Rachford, Henry H},
  journal={Transactions of the American Mathematical Society},
  volume={82},
  number={2},
  pages={421--439},
  year={1956},
  doi={10.1090/S0002-9947-1956-0084194-4},
  publisher={JSTOR}
}

@article{lions1979splitting,
  title={Splitting algorithms for the sum of two nonlinear operators},
  author={Lions, Pierre-Louis and Mercier, Bertrand},
  journal={SIAM Journal on Numerical Analysis},
  volume={16},
  number={6},
  pages={964--979},
  year={1979},
  doi={10.1137/0716071},
  publisher={SIAM}
}

@article{svaiter2011weak,
  title={On weak convergence of the {D}ouglas--{R}achford method},
  author={Svaiter, Benar Fux},
  journal={SIAM Journal on Control and Optimization},
  volume={49},
  number={1},
  pages={280--287},
  year={2011},
  doi={10.1137/100788100},
  publisher={SIAM}
}

@book{bauschke2017correction,
  title={Convex Analysis and Monotone Operator Theory in {H}ilbert Spaces},
  author={Bauschke, Heinz H and Combettes, Patrick L},
  series={CMS Books in Mathematics},
  publisher={Springer International Publishing},
  address={Cham},
  edition={second},
  year={2017},
  doi={10.1007/978-3-319-48311-5}
}

@article{bui2020douglas,
  title={The {D}ouglas--{R}achford Algorithm Converges Only Weakly},
  author={B{\`u}i, Minh N and Combettes, Patrick L},
  journal={SIAM Journal on Control and Optimization},
  volume={58},
  number={2},
  pages={1118--1120},
  year={2020},
  doi={10.1137/19m1308451},
  publisher={SIAM}
}

@article{hundal2004alternating,
  title={An alternating projection that does not converge in norm},
  author={Hundal, Hein S},
  journal={Nonlinear Analysis: Theory, Methods \& Applications},
  volume={57},
  number={1},
  pages={35--61},
  year={2004},
  doi={10.1016/j.na.2003.11.004},
  publisher={Elsevier}
}

@misc{bauschke2026ces,
  title={Ces\`aro means of firmly nonexpansive iterates need not converge strongly},
  author={Bauschke, Heinz H and Tung, Tran Thanh},
  year={2026},
  doi={10.48550/arXiv.2605.25491},
  eprint={2605.25491},
  archivePrefix={arXiv}
}

@misc{bui2026douglas,
  title={The {D}ouglas--{R}achford algorithm for normal cones converges only weakly},
  author={B{\`u}i, Minh N. and Luo, Honglin and Wang, Xianfu},
  year={2026},
  doi={10.48550/arXiv.2610.26147},
  eprint={2610.26147},
  archivePrefix={arXiv}
}

\end{document}